\documentclass[a4poaper,11pt,reqno, english]{amsart}
\usepackage{amssymb, amsthm, amsmath, amsfonts}
\usepackage[left=1in,right=1in,top=1in,bottom=1in]{geometry}
\usepackage[colorlinks=true,urlcolor=blue,linkcolor=red,citecolor=magenta]{hyperref}
\usepackage{cleveref}
\usepackage{subcaption}
\usepackage{graphicx}

\newtheorem{theorem}{Theorem}[section]
\newtheorem{lemma}[theorem]{Lemma}
\newtheorem{corollary}[theorem]{Corollary}
\newtheorem{proposition}[theorem]{Proposition}
\newtheorem{conjecture}[theorem]{Conjecture}

\theoremstyle{definition}
\newtheorem{definition}[theorem]{Definition}
\newtheorem{remark}[theorem]{Remark}
\newtheorem{example}[theorem]{Example}

\DeclareMathOperator{\im}{Im}
\DeclareMathOperator{\mult}{mult}
\DeclareMathOperator{\tr}{tr}
\DeclareMathOperator{\Tr}{Tr}
\DeclareMathOperator{\excludes}{\mathbf{X}}
\newcommand{\graph}[1]{\mathcal{G}_{#1}}
\newcommand{\R}{\mathbb{R}}
\newcommand{\N}{\mathbb{N}}
\newcommand{\Z}{\mathbb{Z}}

\newcommand{\F}{\mathbb{F}}
\newcommand{\exc}[1]{\excludes(#1)}
\newcommand{\set}[1]{\left\{ #1 \right\}}
\newcommand{\restr}[2]{{#1} |_{{#2}}}

\begin{document}

\title{Uniform exclude distributions of Sidon sets}

\author[Thornburgh]{Darrion Thornburgh}

\thanks{Part of this research was done in completion of an undergraduate thesis at Bard College.}

\date{\today}

\begin{abstract}
    A Sidon set $S$ in $\F_2^n$ is a set such that the pairwise sums of distinct points are all distinct.
    The exclude points of a Sidon set $S$ are the sums of three distinct points in $S$, and the exclude multiplicity of a point in $\F_2^n \setminus S$ is the number of such triples in $S$ it is equal to. 
    We call the function $d_S \colon \F_2^n \setminus S \to \Z_{\geq 0}$ taking points in $\F_2^n \setminus S$ to their exclude multiplicity the exclude distribution of $S$.
    We say that $d_S$ is uniform on $\mathcal{P}$ if $\mathcal{P}$ is an equally-sized partition $\mathcal{P}$ of $\F_2^n \setminus S$ such that $d_S$ takes the same values an equal number of times on every element of $\mathcal{P}$.
    In this paper, we use APN plateaued functions with all component functions unbalanced to construct Sidon sets $S$ in $(\F_2^n)^2$ whose exclude distributions are uniform on natural partitions of $(\F_2^n)^2 \setminus S$ into $2^n$ elements.
    We use this result and a result of Carlet to determine exactly what values the exclude distributions of the graphs of the Gold and Kasami functions take and how often they take these values.
\end{abstract}

\maketitle

\section{Introduction}\label{sec:intro}

Sidon sets, first introduced by Simon Sidon \cite{Sidon1932}, are an important notion in additive combinatorics.
In this paper, we consider Sidon sets in the $n$-dimensional vector space over $\F_2$, denoted as $\F_2^n$.
A \textbf{Sidon set} in $\F_2^n$ is a set $S$ such that $a + b = c + d$ has no solutions $(a,b,c,d) \in S^4$ where $a,b,c,d$ are pairwise distinct.
A Sidon set is called \textbf{maximal} if is not contained in any (strictly) larger Sidon sets.
The maximality of a Sidon set can also be determined by its exclude points.

\begin{definition}
    \textup{\cite{quadspaper}}
    Let $S \subseteq \F_2^n$ be a Sidon set.
    The \textbf{exclude points} of $S$ is the set     
    \[
        \exc{S} = \set{a+b+c \in \F_2^n : a,b,c \in S \text{ are pairwise distinct}}
    \]
    Also, the number of distinct triples summing to $x \in \F_2^n \setminus S$ is known as the \textbf{exclude multiplicity} (or the \textbf{multiplicity}) $\mult_S(x)$ of $x$ with respect to $S$.
\end{definition}
Equivalently, the exclude points of a Sidon set $S \subseteq \F_2^n$ is the set of points that, if added to $S$, violate the Sidon property.
So, $S$ is maximal if and only if $\exc{S} = \F_2^n \setminus S$.

In this paper, we study the exclude points of those Sidon sets that are the graphs of almost perfect nonlinear functions.
An \textbf{almost perfect nonlinear} (APN) function is a function $F \colon \F_2^n \to \F_2^n$ such that the equation $F(x+a) + F(x) = b$ has either $0$ or $2$ solutions for any $a,b \in \F_2^n$ where $a \neq 0$.
Note that it is oftentimes convenient to identify $\F_2^n$ with $\F_{2^n}$ to gain a multiplicative structure, and most of the known families of APN functions are constructed over finite fields.
APN functions are studied in cryptography due to their optimal resistance against differential cryptanalysis \cite{NybergBook1994}. 
However, APN functions are also interesting in other areas of research, such as additive combinatorics.
In particular, APN functions are often used in the study of Sidon sets (c.f. \cite{carlet_apnGraphMaximal}, \cite{GNagyThin}, \cite{RedmanRoseWalker}, \cite{taitwon2021}, or \cite{carletPicek}) since a function $F \colon \F_2^n \to \F_2^n$ is APN if and only if its graph $\graph F = \set{(x,F(x)) : x \in \F_2^n}$ is a Sidon set.

It is conjectured that any function obtained by changing the value of an APN function at a single point is not an APN function \cite{budaghyanCarletHellesetUpperBoundsDegree}.
Carlet showed in \cite{carlet_apnGraphMaximal} that this conjecture is equivalent to the following.

\begin{conjecture}\label{conj:allapnfunctionsmaximal}
\textup{\cite{budaghyanCarletHellesetUpperBoundsDegree} \cite{carlet_apnGraphMaximal}}
    The graphs of all APN functions are maximal Sidon sets.
\end{conjecture}

To provide an overview of this paper, in \Cref{sec:prelim} we introduce preliminaries on Sidon sets and APN functions.
This is also the section where we discuss how we visualize $\F_2^n$ in a planar fashion, using a method first introduced in \cite{quadspaper}. 
In \Cref{sec:exclude-dist}, we discuss the exclude distributions of Sidon sets, which are defined as follows.
\begin{definition}
    Let $S \subseteq \F_2^n$ be a Sidon set.
    The \textbf{exclude distribution} of $S$ is the function $d_S \colon \F_2^n \setminus S \to \Z_{\geq 0}$ defined by $d_S(x) = \mult_S(x)$ for any $x \in \F_2^n \setminus S$.
\end{definition}
In particular, we introduce the notions of local equivalence and uniformity of the exclude distribution.
In short, if $S$ is a Sidon set and $X$ and $Y$ are subsets of $\F_2^n \setminus S$ such that $d_S$ takes the same values on $X$ and $Y$ the same number of times, we say that $d_S$ is \textit{locally equivalent} at $X$ and $Y$ (for a more formal definition, see \Cref{sec:exclude-dist}).
Also, if $\mathcal{P}$ is a partition of some set $X \subseteq \F_2^n \setminus S$ such that $d_S$ is locally equivalent at any two elements of $\mathcal{P}$, then we call $d_S$ \textit{uniform} on $\mathcal{P}$.
We then provide examples of uniform exclude distributions.

In \Cref{sec:exclude-dist-of-graphs}, we study the exclude distributions of graphs of APN functions.
In particular, we prove in \Cref{prop:emax-emin-bound-implies-maximal} that if $S \subseteq \F_2^n$ is a Sidon set of size $2^n$ (the same size as the graph of an APN function) such that the difference of the maximal and minimal values that $d_S$ takes is less than or equal to $\frac{2^n -2}{6}$, then $S$ must be maximal.
By using a method which was also used in \cite{carlet_apnGraphMaximal}, we express $d_{\graph F}$ in terms of the Walsh transform.

We denote by $Q_a(F)$ the set $\set{a} \times (\F_2^n \setminus F(a))$ for any $a \in \F_2^n$, and we let $\mathcal{Q}(\F_2^n, F)$ be the collection of all sets $Q_a(F)$.
In \Cref{sec:uniform-exclude-distribution}, we prove the following.

\begin{theorem}\label{thm:LocalEquiv-permutation-implies-maximal}
    Suppose $F \colon \F_2^n \to \F_2^n$ is an APN function such that $F(0) = 0$.
    If $d_{\graph F}$ is locally equivalent at $Q_a(F)$ and $Q_\alpha(F)$ by the permutation $(a,b) \mapsto (\alpha, b + F(a) + F(\alpha))$ for all $a,\alpha \in \F_2^n$, then $\graph F$ is maximal.
\end{theorem}

Following this, we consider an example of the graph of the Gold function and observe that the Gold function has a graph that has an exclude distribution that is uniform on $\mathcal{Q}(\F_{2^n}, F)$. 
We generalize this in \Cref{thm:APNPlateauedComponent-UniformExcludeDist}.

\begin{theorem}\label{thm:APNPlateauedComponent-UniformExcludeDist}
    Let $F \colon \F_2^n \to \F_2^n$ be an APN function.
    If $F$ is a plateaued function whose component functions are all unbalanced, then $d_{\graph F}$ is uniform on $\mathcal{Q}(\F_2^n, F)$.
    If so, then $d_{\graph F}$ is locally equivalent at $Q_a(F)$ and $Q_\alpha(F)$ by the permutation $(a,b) \mapsto (\alpha, b + F(a) + F(\alpha))$ for all $a,\alpha \in \F_2^n$.
\end{theorem}
Informally, if $F$ is an APN plateaued function whose component functions are unbalanced, then $\graph F$ has a very strong symmetry in its exclude points.

Using \Cref{thm:APNPlateauedComponent-UniformExcludeDist} and a result of Carlet, we prove in \Cref{sec:applicationGoldKasami} exactly what values the exclude distributions of the graphs of the Gold and Kasami functions take and how many times they take those values.

\begin{theorem}\label{thm:GoldKasami-MAIN}
    Suppose $n$ is even.
    Suppose $F \colon \F_{2^n} \to \F_{2^n}$ be a Gold function or a Kasami function.
    Let $\alpha(n) = \frac{2^n + (-2)^{\frac{n}{2}+1}- 2}{6}$, and let $\beta(n) = \frac{2^n + (-2)^{\frac{n}{2}} - 2}{6}$
    Then 
    \begin{enumerate}
        \item there are $2^n \cdot \frac{2^n -1}{3}$ exclude points of $\graph F$ with multiplicity $\alpha(n)$;
        \item there are $2^{n+1} \cdot \frac{2^n -1}{3}$ exclude points of $\graph F$ with multiplicity $\beta(n)$;
        \item $d_{\graph F}$ has image $\set{\alpha(n), \beta(n)}$.
    \end{enumerate}
\end{theorem}

This is a hard problem in general, and the exact values that the exclude distribution of the graph of any\footnote{Excluding almost bent (AB) functions (see \Cref{thm:AB-vanDamFlaass}).} other APN functions has yet to be determined.

\section{Background and Preliminaries}\label{sec:prelim}

\subsection{Cryptographic functions}\label{subsec:cryptographic-fcns}
A vectorial Boolean function is a function $F$ from $\F_2^n$ to $\F_2^m$, but in this paper, we only focus on the case $n = m$.
We study the exclude distributions of the graphs $\graph F = \set{(x,F(x)) : x \in \F_2^n}$ of almost perfect nonlinear (APN) functions $F \colon \F_2^n \to \F_2^n$ which are those functions such that $F(x+a) + F(x) = b$ has either $0$ or $2$ solutions for all $a,b \in \F_2^n$ such that $a \neq 0$.
APN functions are an important notion in cryptography as they are those functions that are optimally resistant to a so-called differential attack when used as a substitution box in a block cipher \cite{NybergBook1994}.

It is equivalent to say that a function $F \colon \F_2^n \to \F_2^n$ is APN if and only if $\graph F$ is a Sidon set, that is, a set such that the sum of any pair of distinct elements is distinct.
This connection between Sidon sets and APN functions forms a bridge between additive combinatorics and symmetric cryptography.
For this reason, studying either APN functions or Sidon sets can sometimes lead to results about the other (see, for instance \cite{RedmanRoseWalker} \cite{taitwon2021}). 

Let $F \colon \F_2^n \to \F_2^n$ be a function.
Then $F$ has a unique representation 
\[
    F(x) = \sum_{I \subseteq \set{1, \dots, n}} a_I \prod_{i \in I} x_i
\]
where $a_I \in \F_2^n$, called the \textbf{algebraic normal form} (ANF) of $F$.
The global degree of the ANF of $F$ is called the \textbf{algebraic degree} of $F$.
Moreover, if $F$ has algebraic degree at most $2$ it is called \textbf{quadratic}.
The following conjecture still remains a completely open question.

\begin{conjecture}\label{conj:algDegree}
\textup{\cite{budaghyanCarletHellesetUpperBoundsDegree}}
    No APN function $F \colon \F_2^n \to \F_2^n$ has algebraic degree $n$ for all $n \geq 3$.
\end{conjecture}

It turns out that if \Cref{conj:algDegree} is true, then any function obtained by changing an APN function at a single value is not APN.

\begin{conjecture}\label{conj:changedAtOnePoint}
\textup{\cite{budaghyanCarletHellesetUpperBoundsDegree}}
    Let $F \colon \F_2^n \to \F_2^n$ be an APN function.
    If $G \colon \F_2^n \to \F_2^n$ is a function obtained from $F$ by changing the value of $F$ at one point, then $G$ is not APN.
\end{conjecture}

Moreover, Carlet showed in \cite{carlet_apnGraphMaximal} that \Cref{conj:allapnfunctionsmaximal} and \Cref{conj:changedAtOnePoint} are equivalent.
It seems reasonable that \Cref{conj:allapnfunctionsmaximal} is true as it was shown in \cite{RedmanRoseWalker} that the smallest maximal Sidon set in $\F_2^n$ is of size $O((n \cdot 2^{n})^\frac{1}{3})$ by generalizing a result of Ruzsa \cite{Ruzsa1998}.
However, as mentioned in \cite{carlet_apnGraphMaximal}, 
``there seems to be room for the existence of APN functions whose graphs are non-maximal Sidon sets''
since $|\graph F| = 2^n$ is approximately $\sqrt{2}$ times smaller than the best-known upper bounds on the largest maximal Sidon set (c.f.  \cite{czerwinskiPott2023sidon}).

The \textbf{Walsh transform} of a vectorial Boolean function $F \colon \F_2^n \to \F_2^n$ is the function $W_F \colon (\F_2^n)^2 \to (\F_2^n)^2$ defined by 
\[
    W_F(a,b) = \sum_{x \in \F_2^n} (-1)^{b \cdot F(x) + a \cdot x}
\]
for all $a,b \in \F_2^n$, where $x \cdot y$ denotes the standard inner product.
The Walsh transform is useful as it can be used to characterize many important and desirable cryptographic properties \cite{chabaud_vaudenay_1995}.
Two families of vectorial Boolean functions that we will refer to throughout this paper are plateaued functions and almost bent functions.

\begin{definition}
    Let $F \colon \F_2^n \to \F_2^n$ be a function.
    We call $F$ \textbf{plateaued} if for every $v \in \F_2^n$, there exists $\lambda_v \geq 0$ such that $W_F(u,v) \in \set{0, \pm \lambda_v}$ for all $u \in \F_2^n$.
\end{definition}

\begin{definition}
    Let $F \colon \F_2^n \to \F_2^n$ be a function.
    If $W_F(a,b) \in \set{0, \pm 2^{\frac{n+1}{2}}}$ for all $a,b \in \F_2^n$ such that $(a,b) \neq (0,0)$, then we call $F$ \textbf{almost bent} (AB).
\end{definition}

A \textbf{component function} $v \cdot F \colon \F_2^n \to \F_2$ of a function $F \colon \F_2^n \to \F_2^n$ is given by $(v \cdot F)(x) = v \cdot F(x)$ for all $x \in \F_2^n$, where $v \neq 0$.
There is a more general notion of plateaued functions in terms of component functions of functions from $\F_2^n$ to $\F_2^m$, but our definition coincides with this more general notion for functions $F \colon \F_2^n \to \F_2^n$.

Since the Walsh transform always takes integer values, it follows that AB functions only exist for $n$ odd.
Also, it also immediately follows  observe that all AB functions are plateaued.
Moreover, any AB function is also APN (c.f. \cite{carletCharpinZinovievCodesBentDES} \cite{vandamflass}).
Conversely, it turns out that all APN plateaued functions are also AB if $n$ is odd \cite[Proposition 163]{CarletBook}.
However, not all APN functions are plateaued and not all plateaued functions are APN.

By a result of \cite{vandamflass}, AB functions can be characterized as those APN functions that have a graph with a constant exclude distribution.
\begin{theorem}\label{thm:AB-vanDamFlaass}
\textup{\cite{vandamflass}}
    Let $F \colon \F_2^n \to \F_2^n$ be a function.
    Then $F$ is AB if and only if the system of equations
    \[
        \begin{cases}
            x + y + z = a \\
            F(x) + F(y) + F(z) = b
        \end{cases}
    \]
    has $2^n -2$ or $3 \cdot 2^n - 2$ solutions $(x,y,z) \in (\F_2^n)^3$ for every $(a,b) \in (\F_2^n)^2$.
    If so, then the system has $2^n -2$ solutions if $b \neq F(a)$ and $3 \cdot 2^n - 2$ solutions otherwise.
\end{theorem}
This characterization of AB functions directly implies that any APN function $F \colon \F_2^n \to \F_2^n$ is AB if and only if every point in $(\F_2^n)^2 \setminus \graph F$ has exclude multiplicity $\frac{2^n-2}{6}$ with respect to $\graph F$.
Hence, all AB functions have maximal Sidon sets as their graphs.
Note that it is known that the graphs of APN power functions and the graphs are APN plateaued functions are maximal \cite{budaghyanCarletHellesetUpperBoundsDegree}.

\begin{table}[ht!]
    \centering
    \begin{tabular}{c|c|c|c}
        Name & $d$ & Condition & Reference \\
        \hline
        Gold & $2^k + 1$ & $\gcd(k,n) = 1$ & 
        \cite{GoldR1968} \cite{NybergBook1994}
        \\
        Kasami & $2^{2k} - 2^k + 1$ & $\gcd(k,n) = 1$  & 
        \cite{JanwaWilsonKasami1993} \cite{Kasami1971}
        \\
        Welch & $2^t + 3$ & $n = 2t + 1$ &   
        \cite{Dobbertin1999Welch} 
        \\ 
        Niho & 
        $\begin{cases}
            2^t + 2^{\frac{t}{2}} -1 & \text{if $t$ even} \\
            2^t + 2^{\frac{3t+1}{2}} -1 & \text{if $t$ odd}
        \end{cases}$
        & $n = 2t+1$ & 
        \cite{Dobbertin1999Niho} 
        \\ 
        Inverse & $2^{2t} -1$ & $n = 2t+1$ &   
        \cite{NybergBook1994} \cite{BethDing1994Inverse}
        \\
        Dobbertin & $2^{4t} + 2^{3t} + 2^{2t} + 2^t - 1$ & $n = 5t$ &   
        \cite{Dobbertin2001Dobbertin}
        \\
    \end{tabular}
    \caption{Known infinite families of APN power functions $\F_{2^n} \to \F_{2^n}$ of the form $x \mapsto x^d$.}
    \label{tab:APN-inf-families}
\end{table}

\begin{table}[ht!]
    \centering
    \begin{tabular}{c|c|c|c}
        Name & $d$ & Condition & Reference \\
        \hline
        Gold & $2^k + 1$ & $\gcd(k,n) = 1$ & 
        \cite{GoldR1968} \cite{NybergBook1994}
        \\
        Kasami & $2^{2k} - 2^k + 1$ & $\gcd(k,n) = 1$  & 
        \cite{Kasami1971}
        \\
        Welch & $2^t + 3$ & $n = 2t + 1$ &   
        \cite{WeightDivisCCD2001} \cite{BinarymSequencesCCD2001}
        \\
        Niho & 
        $\begin{cases}
            2^t + 2^{\frac{t}{2}} -1 & \text{if $t$ even} \\
            2^t + 2^{\frac{3t+1}{2}} -1 & \text{if $t$ odd}
        \end{cases}$
        & $n = 2t+1$ & 
        \cite{HollmannXiang2001Niho}
    \end{tabular}
    \caption{Known infinite families of AB power functions $\F_{2^n} \to \F_{2^n}$ of the form $x \mapsto x^d$, $n$ odd.}
    \label{tab:ABinf-families}
\end{table}

In terms of equivalence relations, we call two functions $F,F' \colon \F_2^n \to \F_2^n$ \textbf{CCZ-equivalent} if there exists an affine permutation $\mathcal{A}$ of $(\F_2^n)^2$ such that $\mathcal{A}(\graph F) = \graph{F'}$.
There are very few infinite families of known APN functions, and \Cref{tab:APN-inf-families} and \Cref{tab:ABinf-families} lists the known infinite families of APN power functions and AB power functions, respectively.
Note that these families of APN power functions are up to CCZ-equivalence.

\subsection{Visualizing Sidon sets in \texorpdfstring{$\F_2^n$}{}}\label{sec:qapvis}

A common way to think about $\F_2^n$ is as the vertices of an $n$-dimensional hypercube in $\R^n$.
However, this becomes difficult to do as soon as $n = 4$.
Instead, we visualize $\F_2^n$ in a planar fashion.
The \textit{Qap Visualizer} \cite{QuadsVis} is an online web-based tool used to visualize Sidon sets in $\F_2^n$ where $1 \leq n \leq 14$.
To create a planar representation of $\F_2^n$, we will use a construction that was first introduced in \cite{quadspaper} which is equivalent to the construction used in \cite{QuadsVis}.
First, we start with a planar representation of $\F_2$ as in \Cref{fig:F2-rep}.

\begin{figure}[ht!]
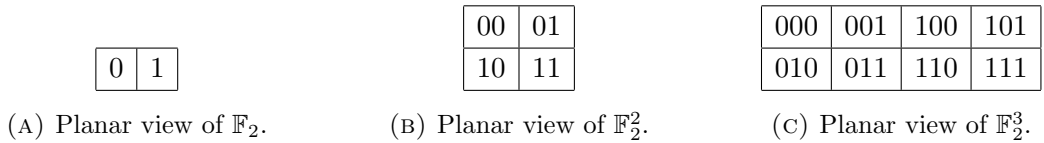

\begin{subfigure}{0.3\textwidth}
    \centering
    \begin{tabular}{|c|c|}
        \hline
         $0$ & $1$ \\
         \hline
    \end{tabular}
    \caption{Planar view of $\F_2$.}
    \label{fig:F2-rep}
\end{subfigure}
\begin{subfigure}{0.3\textwidth}
    \centering
    \begin{tabular}{|c|c|}
        \hline
         $00$ & $01$ \\
         \hline
         $10$ & $11$ \\
         \hline
    \end{tabular}
    \caption{Planar view of $\F_2^2$.}
    \label{fig:F2timesF2-rep}
\end{subfigure}
\begin{subfigure}{0.3\textwidth}
    \centering
    \begin{tabular}{|c|c|c|c|}
        \hline
         $000$ & $001$ & $100$ & $101$ \\
         \hline
         $010$ & $011$ & $110$ & $111$ \\
         \hline
    \end{tabular}
    \caption{Planar view of $\F_2^3$.}
    \label{fig:F2^3-rep}
\end{subfigure}
\caption{Planar views of $\F_2^n$ for $1 \leq n \leq 3$.}
\end{figure}

We then are able to inductively create planar grid layouts of $\F_2^n$ for $n > 1$.
Let us first consider the case when $n > 1$ is even.
If $n > 1$ is even, take two distinct copies of $\F_2^{n-1}$, vertically stack our two grids, and then prepend a $0$ to the vectors in the top half and a $1$ to the vectors in the bottom half. 
Similarly, in the case that $n > 1$ is odd, take two distinct copies of $\F_2^{n-1}$, horizontally stack the two grids, and then prepend a $0$ to the vectors in the left half and a $1$ to the vectors in the right half. 
See \Cref{fig:F2^3-rep}.
From these two steps, we can inductively construct a planar grid layout of $\F_2^n$ that has $2^{\lfloor \frac{n}{2} \rfloor}$ rows and $2^{\lceil \frac{n}{2} \rceil}$ columns for any $n \in \N$.

Using this planar representation of $\F_2^n$, we can visualize Sidon sets in $\F_2^n$ for any $n \in \N$. 
For a Sidon set $S \subseteq \F_2^n$, we picture 
points in $S$ as green diamonds, and the exclude points of $S$ are labeled with their multiplicity.

\begin{figure}[ht!]
    \centering
\begin{subfigure}{0.4\textwidth}
    \centering
    \includegraphics[width=0.4\textwidth]{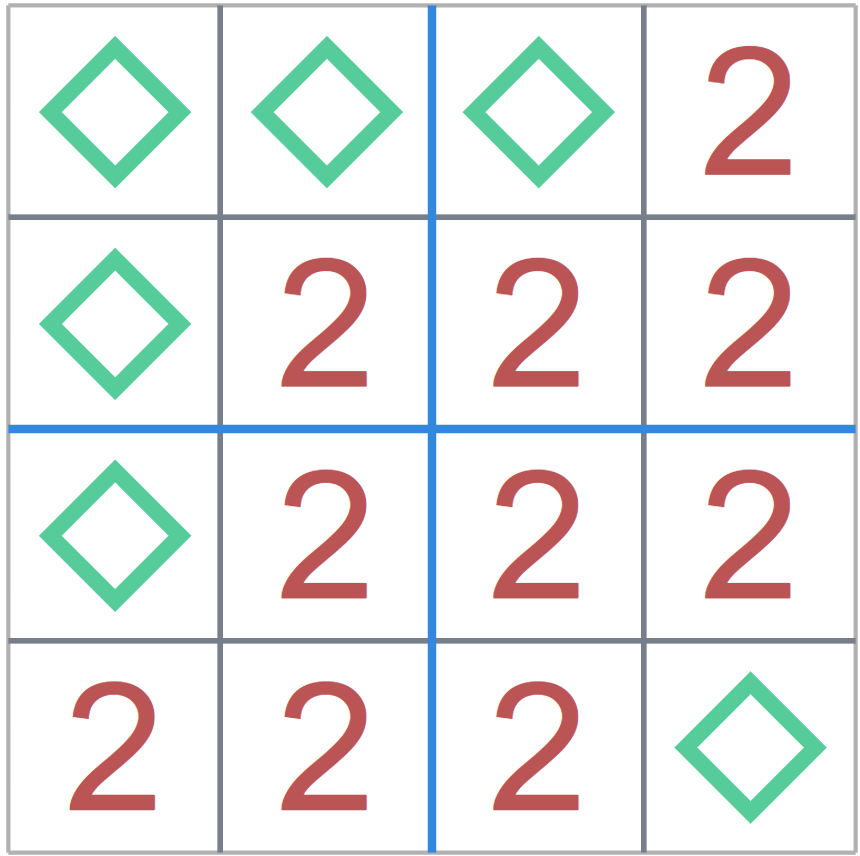}
    \caption{A maximal Sidon set in $\F_2^4$ of size $6$.}
    \label{fig:maxSidonSet-dim4}
\end{subfigure}
\begin{subfigure}{0.4\textwidth}
    \centering
    \includegraphics[width=0.45\textwidth]{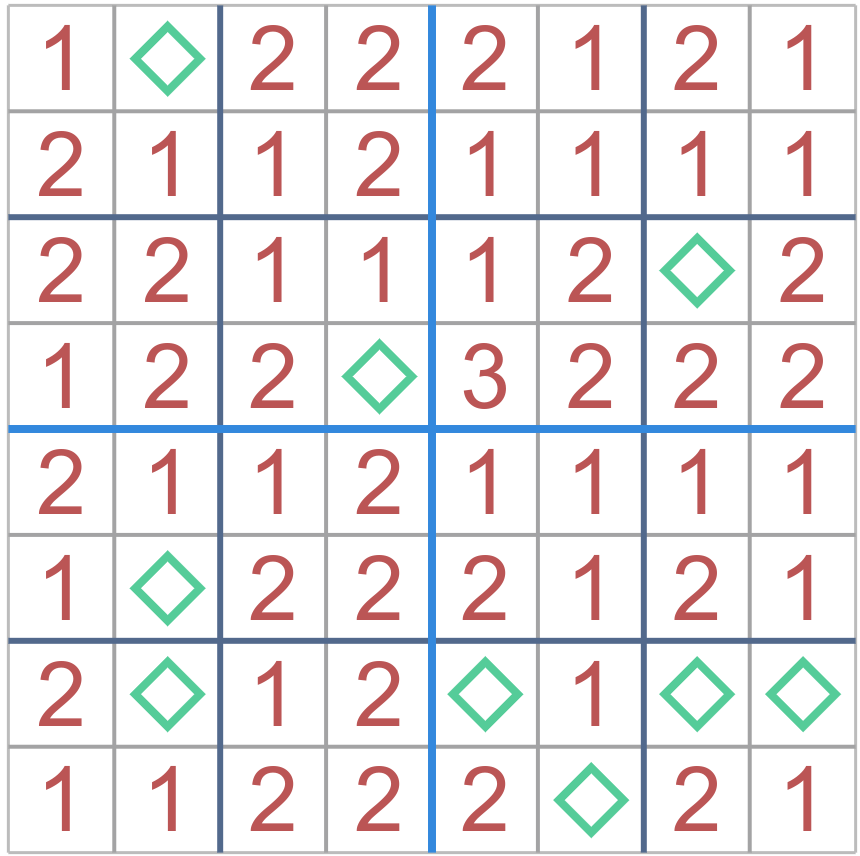}
    \caption{A Sidon set in $\F_2^6$ of size $9$.}
    \label{fig:maxSidonSet-dim6}
\end{subfigure}
    \caption{Two Sidon sets of maximum size for their dimension.}
    \label{fig:maximalSize-Sidonsets}
\end{figure}

\section{The exclude distribution}\label{sec:exclude-dist}

For a Sidon set $S \subseteq \F_2^n$, the function $d_S \colon \F_2^n \setminus S \to \Z_{\geq 0}$ defined by $d_S(x) = \mult_S(s)$ for all $x \in \F_2^n \setminus S$ is called the exclude distribution of $S$.
In this section, we provide some elementary results on exclude distributions in general.
Also, we will introduce some definitions that will be used later on in \Cref{sec:exclude-dist-of-graphs}, where we will study the exclude distributions of graphs of APN functions.
Moreover, in \Cref{prop:emax-emin-bound-implies-maximal}, we prove that any Sidon set of size $2^n$ (the same size as the graphs of APN functions) whose exclude distribution varies in value by at most $\frac{2^n -2}{6}$ must be a maximal Sidon set.

\subsection{Preliminaries on the exclude distribution}
Recall that the exclude set $\exc{S}$ of a Sidon set $S \subseteq \F_2^n$ is exactly the set of points such that any superset of $S$ including a point from $\exc S$ is not a Sidon set.
The following proposition from \cite{quadspaper} relates the size of $S$ and the sum of all the multiplicities of points in $\exc{S}$.

\begin{proposition}\label{prop:sumMults-equal-choose3}
    \textup{\cite{quadspaper}}
    Let $S \subseteq \F_2^n$ be a Sidon set.
    Then $\sum_{x \in \exc{S} } \mult_S(x) = \binom{|S|}{3}$.
\end{proposition}

The complement of a Sidon set $S \subseteq \F_2^n$ is the disjoint union of $\exc S$ and the set of all points of exclude multiplicity $0$ (with respect to $S$).
For this reason, we have the following proposition.
\begin{proposition}\label{prop:sum-exc-equals-eexc}
    Let $S \subseteq \F_2^n$ be a Sidon set.
    Then $\sum_{x \in \exc S} \mult_S(x) = \sum_{x \in \F_2^n \setminus S} \mult_S(x)$.
\end{proposition}
\begin{proof}
    If $x \in \F_2^n \setminus S$, then $x \notin \exc{S}$ if and only if $\mult_S(x) = 0$.
    Therefore, the sums of the multiplicities of points in $\exc{S}$ and $\F_2^n \setminus S$ are equal.
\end{proof}

In other words, for a Sidon set $S$, the sum of the exclude multiplicities of points in $\F_2^n \setminus (\exc{S} \cup S)$ is always zero because all points with non-zero exclude multiplicities are in the excludes of $S$.
Hence, we have the following equality
\begin{equation}\label{eqn:exc-eexc-Schoose3}
     \binom{|S|}{3} 
     = \sum_{x \in \F_2^n \setminus S} \mult_S(x).
\end{equation}

Denote by $e_{\min}(S)$\label{sym:excludeMin} and $e_{\max}(S)$\label{sym:excludeMax} the minimal and maximal exclude point multiplicities, respectively.
That is,  
\begin{align*}
    e_{\min}(S) &= \min_{x \in \F_2^n \setminus S} \mult_S(x), \text{ and} \\
    e_{\max}(S) &= \max_{x \in \F_2^n \setminus S} \mult_S(x).
\end{align*}
We then can use \cref{eqn:exc-eexc-Schoose3} to prove the following proposition.

\begin{proposition}\label{prop:sumMults-equal-choose3-improved}
    Let $S \subseteq \F_2^n$ be a Sidon set, and let $s = |S|$.
    Let $z$ be the number of points in $\F_2^n \setminus S$ with multiplicity $0$.
    Then 
    \begin{equation}\label{eq:inequality-excpts}
        (2^n - s)e_{\min}(S)  \leq \binom{s}{3} \leq  (2^n - s - z)e_{\max}(S).
    \end{equation}
\end{proposition}
\begin{proof}
    By \cref{eqn:exc-eexc-Schoose3}, we have 
    $\binom{s}{3} 
    = \sum_{x \in \F_2^n \setminus S} \mult_S (x) 
    \geq |\F_2^n \setminus S| \cdot e_{\min}(S) 
    =  (2^n - s) e_{\min}(S)$.
    Moreover, we also have
    $\binom{s}{3} = \sum_{x \in \exc S} \mult_S (x) 
        \leq |\exc S| \cdot e_{\max}(S) 
        = (|\F_2^n \setminus S | - z)  e_{\max}(S) 
        = (2^n - s - z) e_{\max}(S).$
    Thus, \cref{eq:inequality-excpts} holds.
\end{proof}

In the case where $e_{\min}(S)$ and $e_{\min}(S)$ are equal, we call $S$ a \textbf{$k$-cover}\footnote{The term $k$-cover was first conceived by Redman, Rose, and Walker \cite{RedmanRoseWalker}.} where $k = e_{\min}(S) = e_{\max}(S)$.
Note that for all $n \in \N$, there does not always exist $k$ such that there is a $k$-cover in $\F_2^n$ \cite{quadspaper}.
Any $k$-cover is a maximal Sidon set if  $k > 0$.
Hence, all non-trivial $k$-covers are maximal Sidon sets.
In some sense, $k$-covers are the most symmetrical Sidon sets and very little is known about them in general.
What is known is that there exist a $(\frac{2^n-2}{6})$-cover for every $n \geq 3$ since $\graph F$ is a $(\frac{2^n-2}{6})$-cover for any AB function $F \colon \F_2^n \to \F_2^n$ by \Cref{thm:AB-vanDamFlaass}.

\begin{proposition}\label{prop:kcovers-constantDist-classification}
    Let $S$ be a Sidon set in $\F_2^n$.
    The following are equivalent:
    \begin{enumerate}
        \item there exists some $k \in \Z_{\geq 0}$ such that $S$ is a $k$-cover,
        \item $e_{\min}(S) = e_{\max}(S)$,
        \item $d_S$ is constant,
    \end{enumerate}
\end{proposition}

So, in some sense, a Sidon set is closer to resembling a $k$-cover if its exclude distribution has much local symmetry. 
We formalize this notion in the following definition.

\begin{definition}
    Let $S$ be a Sidon set in $\F_2^n$.
    Let $X$ and $Y$ be disjoint subsets of $\F_2^n \setminus S$ of the same size.  
    If there exists a permutation $\pi \colon X \to Y$ such that 
    $\restr{d_S}{X} = \restr{d_S}{Y} \circ \pi$, 
    we say that $d_S$ is \textbf{locally equivalent} at $X$ and $Y$.
\end{definition}

Hence, $k$-covers are those Sidon sets whose exclude distribution is locally equivalent at any two equally-sized subsets of their complement.
Hence, the exclude distribution of a $k$-cover $S$ is locally equivalent at any two elements of an equally-sized partition (a partition consisting of elements of the same size) $\mathcal{P}$ of some set $X \subseteq \F_2^n \setminus S$.
We generalize this notion with the following definition.

\begin{definition}
    Let $S$ be a Sidon set in $\F_2^n$.
    If $\mathcal P$ is an equally-sized partition of some set $X \subseteq \F_2^n \setminus S$, then we call $d_S$ \textbf{uniform} on $\mathcal{P}$ if $d_S$ is locally equivalent at any two distinct elements of $\mathcal{P}$.
\end{definition}

\begin{example}
    Suppose $S \subseteq \F_2^n$ is a $k$-cover.
    Let $X \subseteq \F_2^n \setminus S$, and let $\mathcal{P}$ be an equally-sized partition of $X$.
    Then by \Cref{prop:kcovers-constantDist-classification}, $d_S$ is locally equivalent at any two elements of $\mathcal P$, implying $d_S$ is uniform on $\mathcal{P}$.
\end{example}

\begin{example}
    Consider the Sidon set pictured in \Cref{fig:5point-Sidonset} and call it $S$.
    Let $X$ be the highlighted region pictured in \Cref{fig:5point-Sidonset}.
    Notice that $X$ is the union of $6$ distinct $4$-flats (or $4$-dimensional affine subspaces), and let $P_1, \dots, P_6$ be these $4$-flats.
    It is then immediately clear that $d_S$ is locally equivalent at any two of these $4$-flats.
    Therefore, $d_S$ is uniform on $\set{P_1, \dots, P_6}$.
    \begin{figure}[ht!]
        \centering
        \includegraphics[width=0.3\textwidth]{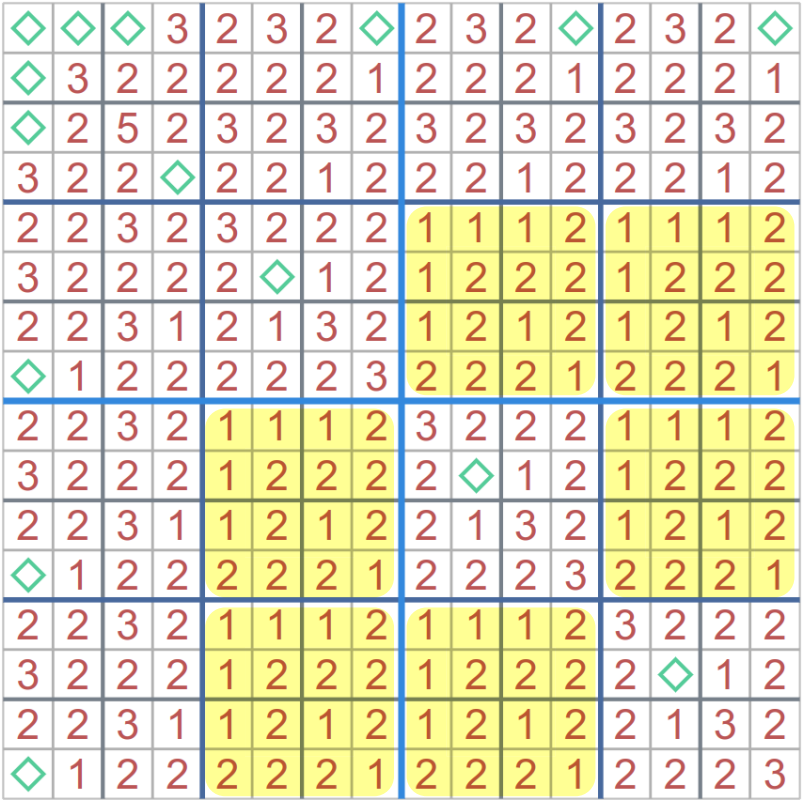}
        \caption{A Sidon set in $\F_2^8$ whose exclude distribution is uniform on $6$ distinct $4$-flats (or $4$-dimensional affine subspaces).}
        \label{fig:5point-Sidonset}
    \end{figure}
\end{example}

Clearly, for any Sidon set $S \subseteq \F_2^n$ and $e_{\min}(S) \leq k \leq e_{\max}(S)$, the exclude distribution of $S$ is uniform on the partition consisting of singleton sets containing points of exclude multiplicity $k$, with respect to $S$.
However, we will only study the cases where the exclude distribution of a Sidon set is uniform on an equally-sized partition of a large set.
We construct such Sidon sets in \Cref{sec:uniform-exclude-distribution}.

\subsection{Equivalence of exclude distributions}
The exclude distribution of two distinct Sidon sets can also be used as an invariant for affine equivalence. 
We call two subsets $S,S' \subseteq \F_2^n$ \textbf{affinely equivalent} if there is an affine permutation $\mathcal{A}$ of $\F_2^n$ such that $\mathcal{A}(S) = S'$.
Sidon sets of size of less than or equal to $9$ are classified up to affine equivalence \cite{quadspaper}.
However, determining whether two Sidon sets are affinely equivalent or not is a difficult problem, in general.
Hence, invariants help determine the affine equivalence of different Sidon sets.

\begin{definition}
    Let $S$ be a Sidon set in $\F_2^n$. If $S' \subseteq \F_2^n$ is a Sidon set, we say that $S$ and $S'$ are \textbf{exclude distribution equivalent} (ED-equivalent) if there exists a permutation $\sigma \colon \F_2^n \setminus S \to \F_2^n \setminus S'$ such that $d_S = d_{S'} \circ \sigma$.
\end{definition}

So, ED-equivalence considers the exclude distributions of two Sidon sets $S$ and $S'$ to be equivalent if and only if the number of points with exclude multiplicity $k$ with respect to $S$ is equal to the number of such points with respect to $S'$.
Equivalently, $S$ and $S'$ are ED-equivalent if and only if $|d^{-1}_S(\set{k})| =|d^{-1}_{S'}(\set{k})|$ for all $k \geq 0$.

\begin{remark}
    Since AB functions are those whose graph is a $(\frac{2^n -2}{6})$-cover, the graphs of any two AB functions (with the same dimension) are ED-equivalent.
    More generally, it follows from \Cref{prop:kcovers-constantDist-classification} that all $k$-covers in the same dimension are ED-equivalent.
\end{remark}

Now, we prove that ED-equivalence of Sidon sets is an invariant of affine equivalence, that is, we show that if two Sidon sets are affinely equivalent, then they must be ED-equivalent.

\begin{proposition}\label{prop:affineequiv-implies-edequiv}
    Let $S, S' \subseteq \F_2^n$ be Sidon sets.
    If $S$ and $S'$ are affinely equivalent, then $S$ and $S'$ are ED-equivalent.
\end{proposition}
\begin{proof}
    Throughout this proof, let $[m] = \set{1, \dots, m}$ for any $m \in \N$.
    Suppose that $S$ and $S'$ are affinely equivalent.
    Then there exists an affine permutation $\mathcal{A} \colon \F_2^n \to \F_2^n$ such that $\mathcal{A}(S) = S'$.
    Clearly, if $|S| = |S'| < 3$, then exclude points in both $\F_2^n \setminus S$ and $\F_2^n \setminus S'$ all have multiplicity $0$, implying that $d_S = d_{S'} \circ \sigma$ where $\sigma \colon \F_2^n \setminus S \to \F_2^n \setminus S'$ is any permutation.
    Hence $S$ and $S'$ are ED-equivalent if $|S| = |S'| < 3$.
    Suppose $|S| \geq 3$.
    Let $x \in \F_2^n \setminus S$, and let $k = \mult_S(x)$.
    \begin{description}
        \item[Case 1] Suppose $k = 0$.
        Then $a_1+a_2+a_3 \neq x$ for all pairwise distinct $a_1,a_2,a_3 \in S$.
        Since $\mathcal{A}$ is a permutation, we know that $\mathcal{A}(a_1 + a_2 + a_3) = \mathcal{A}(x)$ if and only if $a_1 + a_2 + a_3 = x$ for all $a_1, a_2 a_3 \in S$.
        So $\mathcal{A}(a_1) + \mathcal{A}(a_2) + \mathcal{A}(a_3) = \mathcal{A}(a_1+a_2+a_3) \neq \mathcal{A}(x)$ for all pairwise distinct $a_1,a_2,a_3 \in S$.
        Therefore, no three pairwise distinct points in $S' = \mathcal{A}(S)$ sum to $\mathcal{A}(x)$, implying that $\mult_{S'}(\mathcal{A}(x)) = 0$.
        Thus, $d_S = d_{S'} \circ \mathcal{A}$.
        \item[Case 2] Suppose $k > 0$.
        Then, there exist pairwise distinct points $a_1, \dots, a_{3k} \in S$ such that $x = a_i + a_{2i} + a_{3i}$ for all $i \in [k]$.
        Hence $\mathcal{A}(x) =\mathcal{A}(a_i + a_{2i} + a_{3i}) = \mathcal{A}(a_i) + \mathcal{A}(a_{2i}) + \mathcal{A}(a_{3i})$
        for all $i \in [k]$, so $\mult_S(x) \leq \mult_{S'}(\mathcal{A}(x))$.
        By using a similar argument 
        and using the fact that $\mathcal{A}^{-1}$ is affine, we have $\mult_S(x) \geq \mult_{S'}(\mathcal{A}(x))$.
        Therefore $\mult_S(x) = \mult_{S'}(\mathcal{A}(x))$, so $d_S = d_{S'} \circ \mathcal{A}$.
    \end{description}
    Thus, $S$ and $S'$ are ED-equivalent.
\end{proof}

So, ED-equivalence is an affine invariant of Sidon sets.
Note that this implies maximality is preserved by affine equivalence. 
However, ED-equivalence is not a complete invariant of Sidon sets, i.e. there exist Sidon sets that are ED-equivalent but not affinely equivalent.
To see this, we use a result of Dempwolff, but first, we recall the following definition.
Two power functions $F(x)= x^d$ and $F'(x) = x^{d'}$ over $\F_{2^n}$ are called \textbf{cyclotomic equivalent} if there exists $0 \leq i < n$ such that $d \equiv 2^i \cdot d' \mod 2^n - 1$ or, $d \equiv 2^i \cdot d^{-1} \mod 2^n - 1$ when $\gcd(d,2^n -1) = 1$.
Dempwolff proved in \cite{dempwolffCCZCyclotomic} that two APN power functions are CCZ-equivalent if and only if they are cyclotomic equivalent. 
We use this result in the following remark.

\begin{remark}
    Let $F \colon \F_{2^5} \to \F_{2^5}$ be defined by $F(x) = x^3$ for all $x \in \F_{2^5}$, and let $F' \colon \F_{2^5} \to \F_{2^5}$ be defined by $F'(x) = x^7$ for all $x \in \F_{2^5}$.
    Notice that $F$ is a Gold function and $F'$ is a Welch function (see \Cref{tab:ABinf-families}).
    Since $n$ is odd, both $F$ and $F'$ are AB.
    Therefore, $\graph F$ and $\graph{F'}$ are ED-equivalent because AB functions have constant exclude distributions with constant value $\frac{2^n -2}{6}$.
    Notice that, by definition, $\graph F$ and $\graph{F'}$ are affinely equivalent if and only if $F$ and $F'$ are CCZ-equivalent.
    So, it remains to show that $3 \not \equiv 2^i \cdot 7^{-1} \mod 31$ for all $0 \leq i < 5$ because CCZ-equivalence and cyclotomic equivalence are the same for APN power functions.
    First, notice that $7 \cdot 9 = 63 \equiv 1 \mod 31$, so $7^{-1} = 9$ over $\Z_{31}$.
    Now, we compute $2^i \cdot 7^{-1} \equiv 2^i \cdot 9 \mod 31$ for all $0 \leq i < 5$: 
    \begin{align*}
        2^0 \cdot 9 &\equiv 9 \mod 31 \\
        2^1 \cdot 9 &\equiv 18 \mod 31 \\
        2^2 \cdot 9 &\equiv 5 \mod 31 \\
        2^3 \cdot 9 &\equiv 10 \mod 31 \\
        2^4 \cdot 9 &\equiv 20 \mod 31.
    \end{align*}
    Thus, $3 \not \equiv 2^i \cdot 7^{-1} \mod 31$ for all $0 \leq i < 5$, and so $F$ and $F'$ are not CCZ-equivalent, implying that $\graph F$ and $\graph{F'}$ are not affinely equivalent.
    Thus, affine equivalence of Sidon sets is strictly more general than ED-equivalence.
\end{remark}

\section{The exclude distribution of \texorpdfstring{$\graph F$}{}}\label{sec:exclude-dist-of-graphs}

In this section, we study graphs of APN functions and their exclude distributions.
As previously mentioned, it is conjectured that $\graph F$ is maximal for any APN function $F \colon \F_2^n \to \F_2^n$.
In this section, we prove that the difference between the maximal and minimal values that $d_{\graph F}$ takes is at most $\frac{2^n - 2}{6}$, then $\graph F$ is maximal.
Furthermore, we prove that the graph of any APN plateaued function $F \colon \F_2^n\to \F_2^n$ whose component functions are unbalanced is uniform on $\mathcal{Q}(\F_2^n, F)$.
This result highlights a very strong regularity in the exclude multiplicities of $\graph F$.
Moreover, we will see in \Cref{sec:applicationGoldKasami} that this main result allows us to compute the exact values that $d_{\graph F}$ takes and precisely how many times it takes those values when $F$ is a Gold function or Kasami function.

\subsection{The maximal Sidon set conjecture for APN functions}\label{sec:maximalAPNConj}

Recall that a function $F \colon \F_2^n \to \F_2^n$ is APN if and only if its graph $\graph F$ is a Sidon set. 
It has been shown that the graphs of all APN power functions and APN plateaued functions have graphs that are maximal Sidon sets \cite{budaghyanCarletHellesetUpperBoundsDegree}.

To prove maximality of a Sidon set, one can also consider the difference between its minimal and maximal exclude multiplicities.
This is because \Cref{prop:sumMults-equal-choose3-improved} provides a relation that involves the size of the Sidon set $S$, $e_{\min}(S)$ and $e_{\max}(S)$, and also the number of $0$-points in $\F_2^n \setminus S$.
Informally speaking, if the difference between the minimal and maximal exclude multiplicities is small enough, then the Sidon set is ``dense'' which implies that it is maximal.

\begin{proposition}\label{prop:emax-emin-bound-implies-maximal}
    Suppose $n >1$ and $S \subseteq \F_2^{2n}$ is a Sidon set of size $2^n$.
    If 
    \begin{equation}\label{eqn:eMax-eMin-Leq-2n-2over6}
        e_{\max}(S) - e_{\min}(S) \leq \frac{2^n - 2}{6},
    \end{equation}
    then $S$ is maximal.
\end{proposition}
\begin{proof}
    By way of contradiction, suppose $S$ is not maximal.
    Then implies $S$ has an exclude point of multiplicity $0$, so $e_{\min}(S) = 0$.
    Hence, $e_{\max}(S) \leq \frac{2^n-2}{6}$.
    By \Cref{prop:sumMults-equal-choose3-improved}, we have the inequality $\binom{2^n}{3} \leq (2^{2n} - 2^n - 1) e_{\max}(S)$, and since $e_{\max}(S) \leq \frac{2^n-2}{6}$, we have 
    \begin{align*}
        \binom{2^n}{3} \leq (2^{2n} - 2^n - 1) \frac{2^n - 2}{6}.
    \end{align*}
    Observe that this equation is equivalent to 
    \begin{align*}
        \frac{2^n (2^n -1) (2^n -2)}{6} &\leq (2^{2n} - 2^n -1)  \frac{2^n - 2}{6}.
    \end{align*}
    Hence, $2^{2n} - 2^n = 2^n (2^n -1) \leq 2^{2n} - 2^n -1$, a contradiction.
    Thus, $S$ is maximal.
\end{proof}

The converse of \Cref{prop:emax-emin-bound-implies-maximal} does not hold in general, that is, there exist $n \in \N$ and a maximal Sidon set of size $2^n$ in $\F_2^{2n}$ whose exclude distribution takes values varying by more than $\frac{2^n - 2}{6}$.
Also, observe that \Cref{prop:emax-emin-bound-implies-maximal} describes a condition that implies maximality for Sidon sets of size $2^n$ in $\F_2^{2n}$, we can apply this result to the graphs of APN functions as $\F_2^{2n}$ is additively isomorphic to $(\F_2^n)^2$. 
Despite this, an APN function whose graph has an exclude distribution that takes values varying by more than $\frac{2^n-2}{6}$ has yet to be found.

It has been known since \cite{vandamflass}, and perhaps earlier, that sums of subsets of size $3$ of $\graph F$ (i.e. exclude points) are related to the Walsh transform.
The following was shown in \cite[Proof of Cor. 3.2]{carlet_apnGraphMaximal}, and we take it to be a lemma.

\begin{lemma}\label{lem:exclude-mult-Walsh}
    \textup{\cite{carlet_apnGraphMaximal}}
    Let $F \colon \F_2^n \to \F_2^n$ be an APN function.
    Then $\sum_{(u,v) \in (\F_2^n)^2} (-1)^{v \cdot b + u \cdot a} W_F^3(u,v)$ equals
    \[
     2^{2n} |\set{(x_1,x_2,x_3) \in (\F_2^n)^3 : (x_1 +x_2 + x_3, F(x_1) + F(x_2) + F(x_3)) = (a,b)}|.
    \]
    for all $(a,b) \in (\F_2^n)^2$.
\end{lemma}

This allows us to draw a direct connection to the exclude distribution of the graph of an APN function and the Walsh transform. 

\begin{proposition}
\label{prop:excludeMult-InTermsOf-Walsh}
    Let $F \colon \F_2^n \to \F_2^n$ be an APN function.
    If $(a,b) \in (\F_2^n)^2 \setminus \graph F$, then 
    \begin{equation}\label{eq:excludeMult-Walsh}
        d_{\graph F}(a,b) = \frac{1}{3 \cdot 2^{2n+1}} \sum_{(u,v) \in (\F_2^n)^2} (-1)^{v \cdot b + u \cdot a} W_F^3(u,v).
    \end{equation}
\end{proposition}
\begin{proof}
    Let $(a,b) \in (\F_2^n)^2 \setminus \graph F$.
    Since $b \neq F(a)$, we know that there does not exist $(x,y,z) \in (\F_2^n)^3$ such that $\set{x,y,z} < 3$ and $(x+y+z, F(x)+F(y)+F(z)) = (a,b)$.
    Hence
    \begin{align*}
        d_{\graph F}(a,b) &= \frac{1}{6} |\set{(x,y,z) \in (\F_2^n)^3 : |\set{x, y, z}|= 3, (x +y + z, F(x) + F(y) + F(z)) = (a,b)}| \\
        &= \frac{1}{6} |\set{(x,y,z) \in (\F_2^n)^3 : (x +y + z, F(x) + F(y) + F(z)) = (a,b)}|.
    \end{align*}
    By applying \Cref{lem:exclude-mult-Walsh}, we have \cref{eq:excludeMult-Walsh}.
\end{proof}

In general, the exclude distribution of a Sidon set does not have such a closed form, and so the importance of \Cref{prop:excludeMult-InTermsOf-Walsh} is not to be understated. 
Carlet used this to show that the graph of APN function $F$ is maximal if and only if for all $(a,b) \in (\F_2^n)^2$, the inequality $\sum_{(u,v) \in (\F_2^n)^2} (-1)^{v \cdot b + u \cdot a} W_F^3(u,v) \neq 0$ holds.

For a function $F \colon \F_2^n \to \F_2^n$, its graph $\graph F$ has size $2^n$.
So if $F$ is APN and $e_{\max}(\graph F) - e_{\min}(\graph F) \leq \frac{2^n-2}{6}$, then $\graph F$ is a maximal Sidon set by \Cref{prop:emax-emin-bound-implies-maximal}.
However, we can now describe this in terms of the Walsh transform.

\begin{proposition}
    \label{prop:WalshBound-ImpliesMaximal}
     Suppose $n >1$, and suppose $F \colon \F_2^n \to \F_2^n$ is an APN function.
     If 
     \begin{equation}\label{eqn:WalshDiferences-Imply-Maximal}
     \left|
     \sum_{
     \substack{
        (u,v) \in (\F_2^n)^2 
        \\ u\cdot (a+c) \neq v \cdot (b+d)} }
     (-1)^{v \cdot b + u \cdot a} W_F^3(u,v) 
     \right|
     \leq 
    2^{3n-1} - 2^{2n},
    \end{equation}
    holds for all $(a,b), (c,d) \in (\F_2^n)^2 \setminus \graph F$, 
    then $\graph F$ is maximal.
\end{proposition}
\begin{proof}
    Suppose \cref{eqn:WalshDiferences-Imply-Maximal} holds.
    Notice that for any $(a,b), (c,d) \in (\F_2^n)^2$ such that $b \neq F(a)$ and $d \neq F(c)$, we have
    \begin{align*}
         d_{\graph F}(a,b) - d_{\graph F}(c,d) 
         &= 
         \frac{1}{3 \cdot2^{2n+1}} \left( 
         \sum_{(u,v) \in (\F_2^n)^2} (-1)^{v \cdot b + u \cdot a} W_F^3(u,v) 
        -  \sum_{(u,v) \in (\F_2^n)^2} (-1)^{v \cdot d + u \cdot c} W_F^3(u,v)
        \right)
    \end{align*}
    by \Cref{prop:excludeMult-InTermsOf-Walsh}.
    By simplifying, we have 
    \begin{align*}
    d_{\graph F}(a,b) - d_{\graph F}(c,d)  
    &= \frac{1}{3 \cdot2^{2n+1}} \sum_{(u,v) \in (\F_2^n)^2} \left( (-1)^{v \cdot b + u \cdot a} - (-1)^{v \cdot d + u \cdot c}\right) W_F^3(u,v)  \\
        &= \frac{1}{3 \cdot2^{2n+1}}  \sum_{\substack{(u,v) \in (\F_2^n)^2 \\ u\cdot (a+c) \neq v \cdot (b+d) }} 
        \left( 
        (-1)^{v \cdot b + u \cdot a} - (-1)^{v \cdot d + u \cdot c}
        \right) W_F^3(u,v) 
    \end{align*}
    for any $(a,b), (c,d) \in (\F_2^n)^2 \setminus \graph F$.
    If $u \cdot (a+c) \neq v \cdot (b+d)$, then $(-1)^{v \cdot b + u \cdot a} - (-1)^{v \cdot d + u \cdot c} = 2 (-1)^{v \cdot b + u \cdot a}$.
    Hence 
    \begin{align*}
        d_{\graph F}(a,b) - d_{\graph F}(c,d) &= \frac{1}{3 \cdot2^{2n}} 
        \sum_{\substack{(u,v) \in (\F_2^n)^2 \\ u\cdot (a+c) \neq v \cdot (b+d) }} 
        (-1)^{v \cdot b + u \cdot a}  W_F^3(u,v)
    \end{align*}
    for any $(a,b), (c,d) \in (\F_2^n)^2 \setminus \graph F$.
    Therefore
    \begin{align*}
        e_{\max}(\graph F) - e_{\min}(\graph F)
        &= 
        \frac{1}{3 \cdot 2^{2n}}
        \max_{\substack{a,b,c,d \in \F_2^n \\ b \neq F(a), d \neq F(c)}}
        \left|
        \sum_{\substack{(u,v) \in (\F_2^n)^2 \\ u\cdot (a+c) \neq v \cdot (b+d) }} 
        (-1)^{v \cdot b + u \cdot a}  W_F^3(u,v)
        \right| \\
        &\leq \frac{2^{3n-1} - 2^{2n}}{3 \cdot 2^{2n}} \\
        &= \frac{2^n - 2}{6}.
    \end{align*}
    Since $e_{\max}(\graph F) - e_{\min}(\graph F) \leq \frac{2^n - 2}{6}$, the graph of $F$ is maximal by \Cref{prop:emax-emin-bound-implies-maximal}.
\end{proof}

\begin{remark}\label{remark:equivalentToLinearHyperplanes}
    Suppose $F \colon \F_2^n \to \F_2^n$ is an APN function.
    In the case of $\graph F$, 
    \Cref{prop:WalshBound-ImpliesMaximal} is equivalent to stating that $\graph F$ is maximal if 
    \[
             \left|
         \sum_{
         (u,v) \in (\F_2^n)^2 \setminus \mathcal{H}
         } 
         (-1)^{v \cdot b + u \cdot a} W_F^3(u,v) 
         \right|
         \leq 
        2^{3n-1} - 2^{2n}
    \]
    for all linear hyperplanes $\mathcal{H}$ of $(\F_2^n)^2$.
\end{remark}

While there are many APN functions whose graphs have exclude points with multiplicity greater than $\frac{2^n-2}{6}$ (e.g. the Dobbertin function when $n =5$), all of our computed examples (mostly low-dimensional examples of power functions and some quadratics) have satisfied the inequalities from \Cref{prop:emax-emin-bound-implies-maximal} and \Cref{prop:WalshBound-ImpliesMaximal}. 
It would be interesting to find a subclass of APN functions that always satisfy this bound on the difference between the maximal and minimal exclude multiplicities of their graphs, and therefore, a subclass of APN functions whose graphs are maximal. 

\subsection{Graphs of APN functions with uniform exclude distributions}\label{sec:uniform-exclude-distribution}
We now discuss APN functions $F \colon \F_2^n \to \F_2^n$ whose graphs have exclude distributions that exhibit nice properties regarding local equivalence and uniformity.
First, we recall an observation by Dillon: for any APN function $F \colon \F_2^n \to \F_2^n$ and any non-zero $c \in \F_2^n$, there exists a solution $(x,y,z) \in (\F_2^n)^3$ to the equation $F(x) + F(y) + F(z) + F(x+y+z) = c$ \cite{carlet_apnGraphMaximal} \cite[p. 381]{CarletBook}.
In \cite{Taniguchi2023}, a generalization of this property was studied, called the \textbf{D-property}.
Using Dillon's observation, we prove \Cref{thm:LocalEquiv-permutation-implies-maximal}.

\begin{proof}[Proof of \Cref{thm:LocalEquiv-permutation-implies-maximal}]
    Suppose that for any $a,\alpha \in \F_2^n$, the exclude distribution of $\graph F$ is locally equivalent at $Q_a(F)$ and $Q_\alpha(F)$ by the permutation $(a,b) \mapsto (\alpha, b + F(a) +  F(\alpha))$.
    To show that $\graph F$ is maximal, it suffices to show that $d_{\graph F}$ takes only non-zero values on $Q_0(F)$.
    Let $b \in \F_2^n$ such that $b \neq 0$.
    By our hypothesis, we know that $d_{\graph F}(0, b) = d_{\graph F}(\alpha, b + F(\alpha))$ for all $\alpha \in \F_2^n$.
    Equivalently, the number of solutions $(x,y,z) \in (\F_2^n)^3$ to
    \begin{align*}
        \begin{cases}
            x + y + z = \alpha \\
            F(x) + F(y) + F(z) = b + F(\alpha)
        \end{cases}
    \end{align*}
    is constant as $\alpha$ ranges over $\F_2^n$.
    This system of equations is the same as $F(x) + F(y) + F(z) + F(x + y + z) = b$, and by Dillon's observation, we know that there exists a solution $(x,y,z) \in (\F_2^n)^3$ to this equation.
    Therefore, $d_{\graph F}(0, b) > 0$, and so $d_{\graph F}$ only takes non-zero values on $Q_0(F)$.
    By applying the uniformity of $d_{\graph F}$ on $\mathcal{Q}(\F_2^n, F)$, it follows that $\graph F$ is maximal.
\end{proof}

We now introduce a very natural partition of $(\F_2^n)^2$.
For any $a \in \F_2^n$, let $P_a$ denote the set $\set{a} \times \F_2^n = \set{(a,b) : b \in \F_2^n}$.
Clearly, $P_{a_1}$ and $P_{a_2}$ are disjoint if and only if $a_1 = a_2$ for all $a_1, a_2 \in \F_2^n$, so $\set{P_a : a \in \F_2^n}$ partitions  $(\F_2^n)^2$.

\begin{figure}[ht!]
    \centering
    \includegraphics[width=0.35\textwidth]{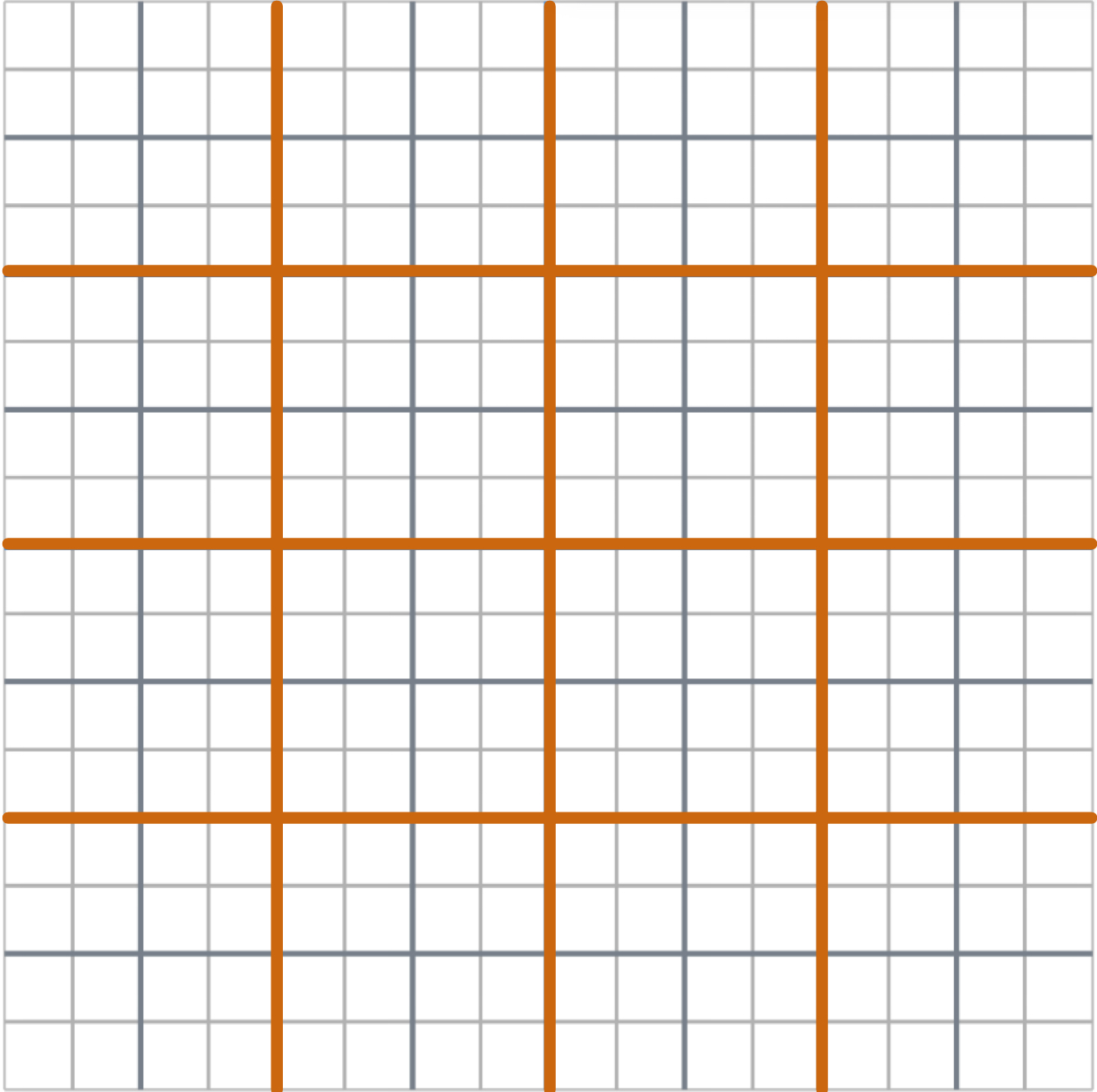}
    \caption{The partition $\set{P_a : a \in \F_2^4}$ of $(\F_2^4)^2$.}
    \label{fig:8dim-grid-partition}
\end{figure}

Notice that if $F \colon \F_2^n \to \F_2^n$ is an APN function, then $(a,b) \in \graph F$ is in $P_\alpha$ if and only if $a = \alpha$.
Hence, there is a natural $1$-to-$1$ correspondence between points in $\graph F$ and the $n$-flats $P_a$.
So, let $Q_a(F)$ be the set $P_a$ with $(a,F(a))$ removed.
Also, let
\begin{align}
    \mathcal{Q}(\F_2^n, F) &= \set{ Q_x(F) : x \in \F_2^n}.
\end{align}
So, $\mathcal{Q}(\F_2^n, F)$ is an equally-sized partition of $(\F_2^n)^2 \setminus \graph F$.

\begin{example}   
    By direct observation of \Cref{fig:goldx3-n4}, we notice that the graph of the Gold function $F(x) = x^3$ over $\F_{2^4}$ has an exclude distribution that is locally equivalent at any $Q_a(F)$ and $Q_\alpha(F)$ for any $a,\alpha \in \F_{2^4}$ since each set $Q_a(F)$ contains $5$ points with exclude multiplicity $1$ and $10$ points with exclude multiplicity $3$.
    In other words, the exclude distribution of $F(x) = x^3$ over $\F_{2^4}$ is uniform on $\mathcal{Q}(\F_{2^4}, F)$.
\end{example}

\begin{figure}[ht!]
    \centering
    \includegraphics[width=0.42\textwidth]{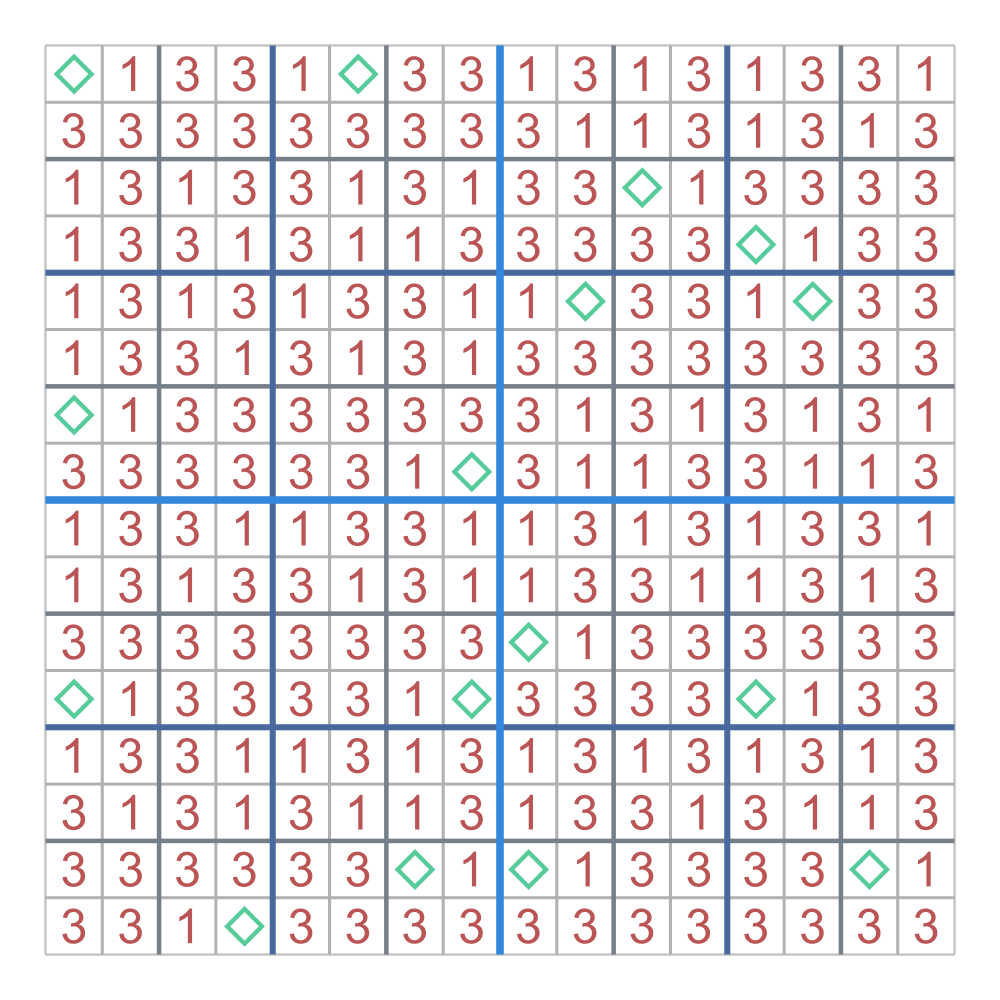}
    \caption{The graph of the Gold function $x \mapsto x^3$ over $\F_{2^4}$.}
    \label{fig:goldx3-n4}
\end{figure}

A natural question to ask which APN functions $F \colon \F_2^n \to \F_2^n$ have the property that $d_{\graph F}$ is uniform on $\mathcal{Q}(\F_2^n, F)$.
Clearly, all AB functions have this property since the exclude distributions of their graphs take constant value.
However, we will soon prove \Cref{thm:APNPlateauedComponent-UniformExcludeDist}, showing that there is a non-trivial family of such APN functions.

Recall that a function $F \colon \F_2^n \to \F_2^n$ is plateaued if and only if, for every $v \in \F_2^n$, there exists $\lambda_v \geq 0$ such that $W_F(u,v) \in \set{0, \pm \lambda_v}$ for all $u \in \F_2^n$.
In the proof Corollary 3 from \cite{carlet_apnGraphMaximal}, it was shown that if $F \colon \F_2^n \to \F_2^n$ is an APN plateaued function whose component functions are all unbalanced, then the following equality holds for every $(a,b) \in (\F_2^n)^2$:
\begin{equation}\label{eqn:APNplateauedfcn-unbalancedcomponenets}
    \sum_{(u,v) \in (\F_2^n)^2} (-1)^{v \cdot b + u \cdot a} W_F^3(u,v) 
    = 2^{2n} |\set{(x,y) \in (\F_2^n)^2 : F(x) + F(y) + F(a) = b}|.
\end{equation}

Originally, this fact was used to prove any APN plateaued function $F \colon \F_2^n \to \F_2^n$ whose component functions are unbalanced satisfies $\im F + \im F = \F_2^n$.
However, we use this fact to show that all APN plateaued functions $F$ whose component functions are unbalanced have graphs whose exclude distributions are uniform on $\mathcal{Q}(\F_2^n, F)$.

\begin{proof}[Proof of \Cref{thm:APNPlateauedComponent-UniformExcludeDist}]
    Suppose that $F$ is an APN plateaued function whose component functions are all unbalanced.
    Then by \Cref{prop:excludeMult-InTermsOf-Walsh} and \cref{eqn:APNplateauedfcn-unbalancedcomponenets}, for any $(a,b) \in (\F_2^n)^2 \setminus \graph F$ we have 
    \begin{align*}
        d_{\graph F}(a,b) &= \frac{1}{3 \cdot 2^{2n+1}} \sum_{(u,v) \in \F_2^n} (-1)^{v \cdot b + u \cdot a} W_F^3(u,v) \\
        &= \frac{1}{6} |\set{(x,y) \in (\F_2^n)^2 : F(x) + F(y) + F(a) = b}|.
    \end{align*}
    Let $a,\alpha, b \in \F_2^n$ such that $b \neq F(a)$, and set $\beta = b + F(a) + F(\alpha)$.
    Then $\beta \neq F(\alpha)$, so $(\alpha, \beta) \notin \graph F$.
    Therefore
    \begin{align*}
        d_{\graph F}(a,b) &= \frac{1}{6} |\set{(x,y) \in (\F_2^n)^2 : F(x) + F(y) + F(a) = b}| \\
        &= \frac{1}{6} |\set{(x,y) \in (\F_2^n)^2 : F(x) + F(y) + F(\alpha) = b +F(a) + F(\alpha)}| \\
        &=  \frac{1}{6} |\set{(x,y) \in (\F_2^n)^2 : F(x) + F(y) + F(\alpha) = \beta}| \\
        &= d_{\graph F}(\alpha, \beta).
    \end{align*}
    We then know that the permutation $\pi_{a,\alpha} \colon Q_a(F) \to Q_\alpha(F)$ given by $(a, b) \mapsto (\alpha, b + F(a) + F(\alpha))$ satisfies $\restr{d_{\graph F}}{Q_a(F)} = \restr{d_{\graph F}}{Q_{\alpha}(F)} \circ \pi_{a,\alpha}$, implying $d_{\graph F}$ is uniform on $\mathcal{Q}(\F_2^n, F)$, as desired.
\end{proof}

We can apply \Cref{thm:APNPlateauedComponent-UniformExcludeDist} to the Gold and Kasami functions in particular. 

\begin{corollary}\label{cor:GoldKasamiUniform}
    Suppose $F \colon \F_{2^n} \to \F_{2^n}$ is a Gold or Kasami function.
    Then $d_{\graph F}$ is uniform on $\mathcal{Q}(\F_{2^n}, F)$. 
\end{corollary}
\begin{proof}
    If $n$ is odd, then $F$ is AB, or in other words, $\graph F$ is a $(\frac{2^n-2}{6})$-cover.
    Hence, $d_{\graph F}$ is uniform on $\mathcal{Q}(\F_{2^n}, F)$ if $n$ is odd by \Cref{prop:kcovers-constantDist-classification}.

    Suppose $n$ is even.
    As proved by Dobbertin, when $n$ is even, any APN power function over $\F_{2^n}$ is $3$-to-$1$ on $\F_{2^n}^\ast$ (see \cite[Proposition 165]{CarletBook} for a proof).
    Observe that for any $v \in \F_{2^n}^\ast$, the component function $v \cdot F$ is unbalanced if and only if $W_F(0,v) \neq 0$.
    Therefore, for any $v \in \F_{2^n}^\ast$,
    \begin{align*}
        W_F(0,v) &= \sum_{x \in \F_{2^n}}(-1)^{v \cdot F(x)} \\
        &= 1 + 3\sum_{y \in \im F} (-1)^{v \cdot y}.
    \end{align*}
    Therefore, all component functions of $F$ are unbalanced.
    
    Moreover, all quadratic functions are plateaued (see for instance \cite{CarletBook}), so if $F$ is a Gold function then it is plateaued because all Gold functions are quadratic.
    Also, $F$ is plateaued if it is a Kasami function because Kasami functions are plateaued when $n$ is even \cite{DillionDobbertin2004NewCyclicDiffSets} and for $n$ coprime with $3$ \cite{YoshiaraKasamiPlateaued}.
    Thus, $d_{\graph F}$ is uniform on $\mathcal{Q}(\F_{2^n}, F)$ by \Cref{thm:APNPlateauedComponent-UniformExcludeDist}.
\end{proof}

Since we can express exclude multiplicity in terms of the Walsh transform, we are also able to express \Cref{thm:APNPlateauedComponent-UniformExcludeDist} in terms of the Walsh transform.

\begin{corollary}
    Suppose $F \colon \F_2^n \to \F_2^n$ is an APN plateaued function whose component functions are all unbalanced. 
    Then, for any $a,\alpha,b \in \F_2^n$ such that $b \neq F(a)$, the equality
    \begin{align*}
        \sum_{(u,v) \in (\F_2^n)^2} (-1)^{u \cdot a + v\cdot b} W_F^3 (u,v)
        = 
        \sum_{(u,v) \in (\F_2^n)^2} (-1)^{u \cdot \alpha + v \cdot (b+F(a)+F(\alpha))} W_F^3 (u,v)
    \end{align*}
    holds. 
    Equivalently, 
    \begin{align*}
        \sum_{
        \substack{(u,v) \in (\F_2^n)^2 \\ 
            u \cdot (a + \alpha)  \neq v \cdot (F(a) + F(\alpha))
            }}
            (-1)^{u \cdot a + v \cdot b} W_F^3(u,v) = 0.
    \end{align*}
    for any $a,\alpha,b \in \F_2^n$ such that $b \neq F(a)$.
\end{corollary}
\begin{proof}
    Recall from \Cref{prop:excludeMult-InTermsOf-Walsh} that $d_{\graph F}(a,b) = \frac{1}{3 \cdot 2^{2n+1}} \sum_{(u,v) \in (\F_2^n)^2} (-1)^{u \cdot a + v \cdot b } W_F^3(u,v)$
    for all $(a,b) \in (\F_2^n)^2 \setminus \graph F$.
    Therefore, for any $(a,b), (c,d) \in (\F_2^n)^2 \setminus \graph F$, we know that $d_{\graph F}(a,b) = d_{\graph F}(c,d)$ if and only if $\sum_{(u,v) \in (\F_2^n)^2} (-1)^{u \cdot a + v \cdot b} W_F^3(u,v) = \sum_{(u,v) \in (\F_2^n)^2} (-1)^{u \cdot a + v \cdot d} W_F^3(u,v)$.

    By \Cref{thm:APNPlateauedComponent-UniformExcludeDist}, we know that $d_{\graph F}(a,b) = d_{\graph F}(\alpha, b + F(a) + F(\alpha))$ for all $a,\alpha,b \in \F_2^n$ where $b \neq F(a)$.
    Therefore, 
    \begin{equation*}\label{eqn:excludeMultsEqualInWalshTerms}
        \sum_{(u,v) \in (\F_2^n)^2} (-1)^{u \cdot a 
 + v \cdot b} W_F^3 (u,v)
        = 
        \sum_{(u,v) \in (\F_2^n)^2} (-1)^{u \cdot \alpha + v \cdot (b + F(a) + F(\alpha))} W_F^3 (u,v)
    \end{equation*}
    for any $a,\alpha,b \in \F_2^n$ such that $b \neq F(a)$.
    Moreover, by rearrangement, we have    \begin{equation*}\label{eqn:excludeMultsEqualInWalshTerms2}
        \sum_{(u,v) \in (\F_2^n)^2} (-1)^{u \cdot a + 
v \cdot b} W_F^3 (u,v) - \sum_{(u,v) \in (\F_2^n)^2} (-1)^{u \cdot \alpha + v \cdot (b + F(a) + F(\alpha))} W_F^3 (u,v) = 0.
    \end{equation*}
    By the same reasoning used in the proof of \Cref{prop:WalshBound-ImpliesMaximal}, the equation above is equivalent to  
    \begin{align*}
        \sum_{
        \substack{(u,v) \in (\F_2^n)^2 \\ 
            u \cdot (a+\alpha) \neq v \cdot (F(a) + F(\alpha))
            }}
            (-1)^{u \cdot a + v \cdot b} W_F^3(u,v) = 0.
    \end{align*}
\end{proof}

Finding more families of APN functions $F \colon \F_2^n \to \F_2^n$ whose graphs admit an exclude distribution that is uniform on the partition $\mathcal{Q}(\F_2^n, F)$ may prove to be difficult. 
It would be interesting to classify all APN functions that admit such a graph.
Additionally, it would be interesting to classify all APN functions that do the same for $\mathcal{Q}^\ast(\F_2^n, F) := \mathcal{Q}(\F_2^n, F) \setminus Q_0(F)$.

\section{An application with the Gold and Kasami functions}\label{sec:applicationGoldKasami}
As we have seen, determining the values of the exclude distribution of the graphs of APN functions is a difficult problem in general.
However, if an APN function $F \colon \F_2^n \to \F_2^n$ has a graph whose exclude distribution is uniform on $\mathcal{Q}(\F_2^n, 
F)$, then the induced symmetry slightly reduces the complexity of this problem.
Moreover, in the case of the Gold and Kasami functions, we can determine precisely what values the exclude distributions of their graphs take. 

Suppose $F \colon \F_2^n \to \F_2^n$ is an APN function.
If $(a,b) \in (\F_2^n)^2 \setminus \graph F$, then $d_{\graph F}(a,b)$ is $\frac{s}{6}$ where $s$ is the number of solutions $(x,y,z) \in (\F_2^n)^3$ to the system of equations
\begin{align*}
\begin{cases}
    x + y + z  = a  \\
    F(x) + F(y) + F(z) = b.
\end{cases}
\end{align*}
By substitution, this system of equations becomes
\[
    F(x) + F(y) + F(x+y+a) = b.
\]
So, in order to compute the possible set of values that $d_{\graph F}$ takes, it suffices to compute the number of solutions to $F(x) + F(y) + F(x+y+a) = b$ as ranges $(a,b)$ across $(\F_2^n)^2 \setminus \graph F$.
If $d_{\graph F}$ is uniform on $\mathcal{Q}(\F_2^n, F)$, then $d_{\graph F}$ is locally equivalent at $Q_0(F)$ and $Q_\alpha(F)$ for any $\alpha \in \F_2^n$.
This means we can assume $a=0$ without loss of generality when we are considering the number of solutions as we range across $b \in \F_2^n$.
So, it suffices the number of solutions to 
\begin{equation}\label{eqn:general-uniform}
    F(x) + F(y) + F(x+y) = b
\end{equation}
as $b$ ranges across $b \in \F_2^n \setminus \set{F(0)}$. 
Carlet showed in an example from \cite[Sec. 6.5.1]{CarletBook} that if $F \colon \F_{2^n} \to \F_{2^n}$ is a Gold function or a Kasami function, where $n$ is even, then the number of solutions $(x,y) \in \F_{2^n}^2$ to \cref{eqn:general-uniform} equals:
\begin{equation}\label{eq:Carlet-solns}
    \begin{cases}
        3 \cdot 2^n - 2 & \text{when } b = 0, \\
        2^n \pm 2^{\frac{n}{2} + 1} - 2 & \text{when $b$ is a cube } (\frac{2^n-1}{3} \text{ cases}), \\
        2^n \mp 2^{\frac{n}{2}} - 2 & \text{when $b$ is not a cube } (2\cdot \frac{2^n-1}{3} \text{ cases}) .
    \end{cases}
\end{equation}
Using Carlet's result, we will be able to compute the exact values that the exclude distributions of the graphs of the Gold and Kasami functions take, but we first prove the following lemma.

\begin{lemma}\label{lem:integer-IFF-0}
    Let $n \in \N$.
    Then 
    \begin{enumerate}
        \item $\frac{2^n + 2^{\frac{n}{2}} - 2}{6} \in \Z$ if and only if $n \equiv 0 \mod 4$;
        \item $\frac{2^n + 2^{\frac{n}{2}+1} - 2}{6} \in \Z$ if and only if $n \equiv 2 \mod 4$.
    \end{enumerate}
\end{lemma}
\begin{proof}
    Since $2^{\frac{n}{2}}$ is irrational for all odd $n$, it is clear that $\frac{2^n + 2^{\frac{n}{2}} - 2}{6}$ and $\frac{2^n + 2^{\frac{n}{2}+1} - 2}{6}$ are never integers when $n$ is odd.
    So, we only consider the cases when $n$ is even.

    Suppose that $n \equiv 0 \mod 4$.
    Then there exists some $m \in \Z$ such that $n = 4m$.
    For any natural number $k$, it is clear that $2^k$ is congruent to either $2$ or $1$ modulo $3$, depending on whether $k$ is odd or even, respectively.
    For this reason, we know that $2^k - 1 \mod 3$ is $1$ if $n$ is odd and $0$ otherwise.
    Hence, $2^{4m-1} -1 \equiv 1 \mod 3$ and $2^{2m-1} \equiv 2 \mod 3$.
    This implies that $2^{4m-1}+2^{2m-1} -1$ is divisible by $3$.
    So
    \begin{align*}
        \frac{2^n + 2^{\frac{n}{2}} - 2}{6} &= \frac{2^{4m-1} + 2^{2m-1} - 1}{3} 
    \end{align*}
    is an integer.
    Additionally, we know that $2^{2m} \equiv 1 \mod 3$, implying that
    \begin{align*}
        \frac{2^n + 2^{\frac{n}{2}+1} - 2}{6} &= \frac{2^{4m-1} + 2^{2m} - 1}{3}
    \end{align*}
    is not an integer.

    Now, suppose that $n \equiv 2 \mod 4$.
    Then there exists $m \in \Z$ such that $n = 4m + 2$.
    Then $2^{4m+1} - 1\equiv 1 \mod 3$ and $2^{2m} \equiv 1 \mod 3$.
    This directly implies that $2^{4m+1} + 2^{2m} - 1$ is not divisible by $3$.
    Hence
    \begin{align*}
        \frac{2^n + 2^{\frac{n}{2}} - 2}{6} &= \frac{2^{4m+1} + 2^{2m} - 1}{3}
    \end{align*}   
    is not an integer.
    Moreover, we know that $2^{2m+1} \equiv 2 \mod 3$, so 
    \begin{align*}
        \frac{2^n + 2^{\frac{n}{2}+1} - 2}{6} &= \frac{2^{4m+1} + 2^{2m+1} - 1}{3} 
    \end{align*}
    is an integer.
    Thus, both (1) and (2) hold.
\end{proof}

We now characterize the exclude distributions of the graphs of the Gold and Kasami functions.

\begin{proof}[Proof of \Cref{thm:GoldKasami-MAIN}]
    Recall that $d_{\graph F}$ is uniform on $\mathcal{Q}(\F_{2^n}, F)$ by \Cref{cor:GoldKasamiUniform}.
    By using what we have discussed and Carlet's result from \cref{eq:Carlet-solns}, we deduce that there are $2^n \cdot \frac{2^n - 1}{3}$ exclude points of $\graph F$ of multiplicity $\frac{1}{6}(2^n \pm 2^{\frac{n}{2} + 1} - 2)$ and there are $2^{n+1} \cdot \frac{2^n-1}{3}$ exclude points of $\graph F$ of multiplicity $\frac{1}{6}(2^n \mp 2^{\frac{n}{2}} - 2)$.
    However, we can remove the ``$\pm$'' term by applying \Cref{lem:integer-IFF-0}.
    Hence, there are $2^n \cdot \frac{2^n - 1}{3}$ exclude points of $\graph F$ of multiplicity $\frac{2^n + 2^{\frac{n}{2}+1}-2}{6}$ or $\frac{2^n - 2^{\frac{n}{2}+1}-2}{6}$, when $n \equiv 2 \mod 4$ or $n \equiv 0 \mod 4$, respectively. 
    In other words, there are $2^n \cdot \frac{2^n -1}{3}$ points in $\F_{2^n}^2 \setminus \graph F$ that map to $\alpha(n)$ under $d_{\graph F}$, so (1) holds.
    Moreover, there are $2^{n+1} \cdot \frac{2^n - 1}{3}$ exclude points of $\graph F$ of multiplicity $\frac{2^n + 2^{\frac{n}{2}}-2}{6}$ or $\frac{2^n - 2^{\frac{n}{2}}-2}{6}$, when $n \equiv 0 \mod 4$ or $n \equiv 2 \mod 4$, respectively.
    Similarly, there are $2^{n+1}  \cdot \frac{2^n -1}{3}$ points in $\F_{2^n}^2 \setminus \graph F$ that map to $\beta(n)$ under $d_{\graph F}$, so (2) holds.
    Finally, (3) holds because the size of $d^{-1}_{\graph F}(\set{\alpha(n), \beta(n)})$ is $2^{2n} - 2^n = |\F_{2^n}^2 \setminus \graph F|$.
\end{proof}

\section{Future work and computational results}\label{sec:futureWork}

Finding APN functions $F \colon \F_2^n \to \F_2^n$ such that $d_{\graph F}$ is uniform on $\mathcal{Q}(\F_2^n, F)$ are particularly interesting. 
This is because if $F$ is such a function, then $\graph F$ is non-maximal if and only if $\graph F$ has at least $2^n$ points with exclude multiplicity $0$.
So, any APN function whose graph has an exclude distribution that is uniform on this partition has to meet a much stronger condition to be non-maximal.
This motivates the following conjecture.
\begin{conjecture}\label{conj:uniform-implies-maximal}
    Suppose $F \colon \F_2^n \to \F_2^n$ is an APN function.
    If $d_{\graph F}$ is uniform on $\mathcal{Q}(\F_2^n, F)$, then $\graph F$ is maximal.
\end{conjecture}

Also, our computer calculations suggest that power functions always have graphs whose exclude distributions take value $\frac{2^n -2}{6}$ at points of the form $(a,0)$ and $(0,b)$ where $a \in \F_2^n$ and $b \in \F_2^n$.

\begin{conjecture}\label{conj:APN-power-ExcludeConst-on-zero-flat}
    Suppose $n$ is odd.
    Let $F \colon \F_{2^n} \to \F_{2^n}$ be a power function $F(x) = x^d$.
    If $F$ is APN, then $d_{\graph F}(a,0) = d_{\graph F}(0,b) = \frac{2^n-2}{6}$ for any $a \in \F_2^n$ and any $b \in \F_2^n$ such that $b \neq 0$.
\end{conjecture}

Clearly, all AB functions satisfy this conjecture by \Cref{thm:AB-vanDamFlaass}.
Despite this, power functions that are not AB such as the Inverse and Dobbertin appear to satisfy \Cref{conj:APN-power-ExcludeConst-on-zero-flat} for low values of $n$.

\begin{table}[ht!]
    \centering
    \begin{tabular}{c|c|c|c}
    $n$ & Function & $d_{\graph F}$ uniform on $\mathcal{Q}(\F_{2^n}, F)$ & $d_{\graph F}$ uniform on $\mathcal{Q}^\ast(\F_{2^n}, F)$ \\
    \hline
    $4$ & Gold & True & True \\
    $4$ & $x^3 + a^{-1}\tr_n(a^3 x^9)$  & True & True \\
    $5$ & Inverse & False & True \\
    $5$ & Dobbertin & False & True \\
    $6$ & Gold & True & True \\
    $6$ & $x^3 + a^{-1}\tr_n(a^3 x^9)$ & True & True \\
    $6$ & $x^3 + a^{-1} \Tr_3^n(a^3 x^9 + a^6 x^{18})$ & True & True \\
    $6$ & Permutation from \cite{browningDillonMcQuisttanWolfe} & True & True \\
    $7$ & Inverse & False & True \\
    $8$ & Gold  & True  &  True\\
    $8$ & $x^3 + a^{-1}\tr_n(a^3 x^9)$ & True & True  \\ 
    $9$ & Inverse & False & True \\
    $10$ & Gold & True  & True \\
    $10$ & Dobbertin & False & True  \\
    $10$ & $x^3 + a^{-1}\tr_n(a^3 x^9)$ & True & True  
    \end{tabular}
    \caption{APN functions $F \colon \F_{2^n} \to \F_{2^n}$ such that $d_{\graph F}$ is uniform on $\mathcal{Q}(\F_{2^n}, F)$ or $\mathcal{Q}^\ast(\F_{2^n}, F)$, excluding AB functions.}
    \label{tab:computed-examples-uniformExcludeDist}
\end{table}

The only known APN permutation for $n$ even is when $n = 6$.
In a breakthrough result of \cite{browningDillonMcQuisttanWolfe}, it was shown that if $\alpha$ is a primitive element of $\F_{2^6}$, then 
\begin{align*}
    F(x) =
    &\alpha^{25} x^{57}+ \alpha^{30} x^{56}+
    \alpha^{32} x^{50}+\alpha^{37} x^{49}+
    \alpha^{23} x^{48}+\alpha^{39} x^{43}+
    \alpha^{44} x^{42}+\alpha^{4} x^{41}+
    \alpha^{18} x^{40}+\\&\alpha^{46} x^{36}+
    \alpha^{51} x^{35}+\alpha^{52} x^{34}+
    \alpha^{18} x^{33}+\alpha^{56} x^{32}+
    \alpha^{53} x^{29}+\alpha^{30} x^{28}+ 
    \alpha^{1} x^{25}+\alpha^{58} x^{24}+ \\
    &\alpha^{60} x^{22}+\alpha^{37} x^{21}+
    \alpha^{51} x^{20}+\alpha^{1} x^{18}+
    \alpha^{2} x^{17}+\alpha^{4} x^{15}+
    \alpha^{44} x^{14}+\alpha^{32} x^{13}+
    \alpha^{18} x^{12}+\\
    &\alpha^{1} x^{11}+\alpha^{9} x^{10}+
    \alpha^{17} x^{8}+\alpha^{51} x^{7}+
    \alpha^{17} x^{6}+\alpha^{18} x^{5}+
    \alpha^{0} x^{4}+\alpha^{16} x^{3}+
    \alpha^{13} x^{1}
\end{align*}
is an APN permutation over $\F_{2^6}$.
Interestingly, we observed through computer calculations that this permutation has a graph whose exclude distribution $d_{\graph F}$ is uniform on $\mathcal{Q}(\F_{2^6}, F)$.

We list in \Cref{tab:computed-examples-uniformExcludeDist} functions $F$ that are not AB but have graphs whose exclude distributions are uniform on $\mathcal{Q}(\F_{2^n}, F)$ or $\mathcal{Q}^\ast(\F_{2^n}, F) = \mathcal{Q}(\F_{2^n}, F) \setminus Q_0(F)$.
Note that in \Cref{tab:computed-examples-uniformExcludeDist} we use the trace functions $\Tr_3^n(x)$ and $\tr_n(x)$, which are defined to be $\Tr_3^n(x) = x 
+ x^8 + x^{8^2} + \cdots + x^{8^{\frac{n}{3}-1}}$ and $\tr_n(x) = x + x^2 + \cdots + x^{2^n-1}$.

We have yet to find an example of an APN function $F \colon \F_2^n \to \F_2^n$ such that $d_{\graph F}$ is not uniform on $\mathcal{Q}^\ast(\F_2^n, F)$.
This remains an open problem, although it may prove difficult to find such a function.

\section*{Acknowledgements}
The author thanks Robert McGrail and Steven Simon for guidance and advice throughout the work of this research.
The author thanks Lauren Rose for first introducing him to the beautiful mathematics of Sidon sets.
The author also thanks Claude Carlet for helpful references and feedback.

\bibliographystyle{plain}

\end{document}